\documentclass{article}
\usepackage{amsfonts}
\usepackage{amsmath}
\usepackage{amsthm}
\usepackage{amssymb}
\usepackage{geometry}
\usepackage{bbm}
\usepackage{mathtools}
\usepackage{babel}
\usepackage{caption}
\usepackage{tikz}
\usepackage{tkz-graph}
\title{Subordinacy Theory on Star-Like Graphs}
\author{Netanel Levi\footnote{Institute of Mathematics, The Hebrew University of Jerusalem, Jerusalem, 91904, Israel.
		Supported in part by the Israel Science Foundation (Grant No.\ 1378/20) and in part by the United States-Israel Binational Science Foundation
		(Grant No.\ 2020027), Email: netanel.levi@mail.huji.ac.il}}
\numberwithin{equation}{section}
\numberwithin{figure}{section}
\theoremstyle{plain}
\newtheorem{theorem}{\protect\theoremname}[section]
\theoremstyle{definition}
\newtheorem{definition}[theorem]{\protect\definitionname}
\theoremstyle{definition}
\newtheorem*{definition*}{\protect\definitionname}
\theoremstyle{remark}
\newtheorem{remark}[theorem]{\protect\remarkname}
\theoremstyle{plain}
\newtheorem{lemma}[theorem]{\protect\lemmaname}
\theoremstyle{plain}

\theoremstyle{plain}
\newtheorem{prop}[theorem]{\protect\propositionname}
\theoremstyle{plain}
\newtheorem{claim}[theorem]{\protect\claimname}

\DeclareMathOperator{\im}{Im}

\DeclareMathOperator{\tr}{tr}
\DeclareMathOperator{\diag}{diag}
\DeclareMathOperator{\vspan}{sp}
\DeclareMathOperator{\rank}{rank}

\tikzstyle{vertex}=[circle, draw, fill=black, inner sep=0pt, minimum size=6pt]
\newcommand{\vertex}{\node[vertex]}

\makeatother

\providecommand{\claimname}{Claim}
\providecommand{\corollaryname}{Corollary} 
\providecommand{\definitionname}{Definition}
\providecommand{\lemmaname}{Lemma}
\providecommand{\propositionname}{Proposition}
\providecommand{\remarkname}{Remark}
\providecommand{\theoremname}{Theorem}

\begin{document}
\maketitle
\sloppy
\begin{abstract}
	We study Jacobi matrices on star-like graphs, which are graphs that are given by the pasting of a finite number of half-lines to a compact graph. Specifically, we extend subordinacy theory to this type of graphs, that is, we find a connection between asymptotic properties of solutions to the eigenvalue equations and continuity properties of the spectral measure with respect to the Lebesgue measure. We also use this theory in order to derive results regarding the multiplicity of the singular spectrum.
\end{abstract}
\section{Introduction}
In this work, we are interested in studying spectral properties of Jacobi matrices on certain types of graphs. Throughout, we assume that the number of vertices of the graph is infinite. Given a graph $G=\langle V,E\rangle$ where $V$ is the set of vertices and $E$ is the set of edges, a Jacobi matrix $J$ on $G$ is given by two sets of real numbers, $\left\{b_v\right\}_{v\in V}$ and $\left\{a_e\right\}_{e\in E}$. $J$ acts on $\ell^2\left(G\right)$ by
\begin{equation}\label{op_eq}
	J(\varphi)(u)=\sum\limits_{w\sim u}a_{\left(u,w\right)}\varphi_w+b_u\varphi_u.
\end{equation}
For simplicity, we assume throughout that $\left\{b_v\right\}_{v\in V}$ and $\left\{a_e\right\}_{e\in E}$ are bounded, and in that case $J$ is a bounded self-adjoint operator. By the spectral theorem (see, e.g., \cite[Chapter 6]{BS}), $J$ gives rise to a projection-valued measure $P$ on $\sigma(H)$, which is called {\it the spectral measure} of $J$. Our goal is to study continuity properties of $P$ via the asymptotics of solutions to the eigenvalue equation. Namely, given a Jacobi matrix $J$ and $E\in\mathbb{R}$, we study solutions to the formal difference equation
\begin{equation}\label{ev_eq}
	J\varphi=E\varphi,
\end{equation}
where by formal we mean that $\varphi$ need not be in $\ell^2\left(G\right)$. We will distinguish between supports of the absolutely continuous and singular parts of $P$ (with respect to the Lebesgue measure) by studying the asymptotic properties of these solutions.

The method extended in this work is known as subordinacy theory. It was developed in \cite{GP} for the case of continuum Schro\"{o}dinger operators on a half-line. In that case, continuity properties of the spectral measure are determined by comparing the growth of solutions which satisfy different boundary conditions to the differential equation associated with the operator. The discrete analogue of subordinacy theory was later developed to Jacobi matrices on $\mathbb{N}$ in \cite{KP}. In the latter case, the operator is simply given by a tridiagonal matrix whose entries are bounded and real-valued. Given such an operator $J$ and $\theta\in\left[0,\pi\right)$, the operator $J_\theta$ is defined by setting $J_\theta=J-\tan\left(\theta\right)\langle\delta_1,\cdot\rangle\delta_1$. The method of subordinacy enables one to examine continuity properties of the spectral measure of these operators by comparing formal solutions to the equation
\begin{equation}\label{ev_eq_1dim}
	J_{\theta}\varphi=E\varphi,\,\,\,\theta\in\left[0,\pi\right),\,\,\,E\in\mathbb{R}.
\end{equation}
Given $\theta\in\left[0,\pi\right)$, a solution $\varphi$ to (\ref{ev_eq_1dim}) can also be regarded as a solution to the same equation with $\theta=0$, along with the boundary condition
\begin{equation}\label{bcon_eq}
	\varphi_0\cos\theta+\varphi_1\sin\theta=0.
\end{equation}
Throughout, the solution to (\ref{ev_eq_1dim}) with $\theta=0$ will be referred to as the solution which satisfies a Dirichlet boundary condition.

We now turn to briefly describe the method of subordinacy. Let $\theta\in\left[0,\pi\right)$. Note that $\delta_1$ is a cyclic vector for $J_\theta$, namely \mbox{$\ell^2(\mathbb{N})=\overline{\vspan\left\{\delta_1,J_\theta\delta_1,J_\theta^2\delta_1,\ldots\right\}}$}, and so the spectral measure of $J_\theta$ is equivalent to that of $\delta_1$, in the sense that they have the same null sets.
Given $u:\mathbb{N}\to\mathbb{R}$ and $L>0$, denote
\begin{center}
	$\|u\|_L\coloneqq\left[\sum\limits_{n=1}^{[L]}\left|u_n\right|^2+(L-[L])\left|u_{[L]+1}\right|^2\right]^{1/2}$.
\end{center}
\begin{definition}
A solution $\psi$ to (\ref{ev_eq_1dim}) will be called {\it subordinate} if for any other linearly independent solution $\varphi$,
\begin{equation}\label{sub_eq}
	\underset{L\rightarrow\infty}{\lim}\frac{\|\psi\|_L}{\|\varphi\|_L}=0.
\end{equation}
\end{definition}
Denote by $\mu_\theta$ the spectral measure of $\delta_1$, and by $\left(\mu_\theta\right)_s,\left(\mu_\theta\right)_{ac}$ its singular and absolutely continuous parts (with respect to the Lebesgue measure) respectively. In \cite{KP}, it is proved that $\left(\mu_\theta\right)_s$ is supported on the set of energies for which the solution to (\ref{ev_eq}) is subordinate, and $\left(\mu_\theta\right)_{ac}$ is supported on the set of energies for which no subordinate solution exists. The theory was further developed in various directions. In \cite{G2}, it was extended to continuum Schr\"{o}dinger operators on $\mathbb{R}$. In \cite{JL1}, Jitomirskaya and Last present a strengthening of the theory in the discrete half-line case (i.e.\ Jacobi matrices on $\mathbb{N}$), and use it to relate the asymptotic properties of the solutions to continuity properties of the spectral measure with respect to various Hausdorff measures. Subordinacy is a very powerful tool as in many cases, the study of asymptotic properties of solutions to (\ref{ev_eq}) is more accessible, compared to classical methods such as the direct study of the Borel transform of $\mu$. It has many applications and generalizations, \cite{Bu,DKL,G1,JL2,KLS,Rem,Z} is a very partial list.

In this work, we extend the theory of subordinacy to a certain kind of graphs which we call {\it star-like}. Roughly speaking, a star-like graph $G$ is a graph which consists of a compact component $C$ along with a finite collection of half-lines attached to it. In Figure \ref{sg_fig} we give a few examples of star-like graphs. Although the compact component is not unique, our results do not depend on its choice. Thus, throughout we fix a compact component $C$ and refer to it as {\it the} compact component of $G$. For every $v\in C$, we denote by $G_v$ the half-line attached to $v$. If $v$ has no half-line attached to it, then $G_v=\left\{v\right\}$. Let $J$ be an operator of the form (\ref{op_eq}) on $\ell^2(G)$. We say that a non-trivial solution $\psi$ to (\ref{ev_eq}) is subordinate if and only if it is subordinate as a solution on $G_v$ for every $v\in C$ such that $G_v$ is a half-line.

\begin{figure}
	\begin{center}
		\begin{tikzpicture}[scale=0.9]
			\vertex(t1) at (-5,1) {};
			\vertex(t2) at (-3,1) {};
			\vertex(t3) at (-4,-1) {};
			\vertex(t21) at (-2,1.66) {};
			\vertex(t11) at (-6,1.66) {};
			\vertex(t31) at (-4,-2.06) {};
			\Edge (t1)(t2);
			\Edge (t2)(t3);
			\Edge(t3)(t1);
			\Edge(t2)(t21);
			\Edge(t1)(t11);
			\Edge(t3)(t31);
			\draw (-4,-2.06) -- (-4,-3.06)[dashed];
			\draw (-2,1.66) -- (-1,2.33)[dashed];
			\draw(-6,1.66) -- (-7,2.33)[dashed];
			\node at (-4,3) {(a)};

			\vertex(so) at (3,0) {};
			\vertex(s1) at (4,1) {};
			\vertex(s2) at (3,-1) {};
			\vertex(s3) at (2,1) {};
			\Edge (so)(s1);
			\Edge (so)(s2);
			\Edge (so)(s3);
			\draw (4,1) -- (5,2)[dashed];
			\draw (3,-1) -- (3,-2)[dashed];
			\draw (2,1) -- (1,2)[dashed];
			\node at (3,3) {(b)};
			
			\vertex(ro) at (-0.5,-9) {};
			\vertex(r1) at (-0.5,-6.5) {};
			\vertex(r11) at (-0.5,-5.5) {};
			\vertex(r12) at (0.5,-6.5) {};
			\vertex(r13) at (-1.5,-6.5) {};
			\vertex(r2) at (2,-9) {};
			\vertex(r21) at (2,-8) {};
			\vertex(r22) at (3,-9) {};
			\vertex(r23) at (2,-10) {};
			\vertex(r3) at (-0.5,-11.5) {};
			\vertex(r31) at (0.5,-11.5) {};
			\vertex(r32) at (-0.5,-12.5) {};
			\vertex(r33) at (-1.5,-11.5) {};
			\vertex(r4) at (-3,-9) {};
			\vertex(r41) at (-3,-10) {};
			\vertex(r42) at (-4,-9) {};
			\vertex(r43) at (-3,-8) {};
			\Edge (ro)(r1);
			\Edge(r1)(r11);
			\Edge(r1)(r12);
			\Edge(r1)(r13);
			\Edge(ro)(r2);
			\Edge(r2)(r21);
			\Edge(r2)(r22);
			\Edge(r2)(r23);
			\Edge(ro)(r3);
			\Edge(r3)(r31);
			\Edge(r3)(r32);
			\Edge(r3)(r33);
			\Edge(ro)(r4);
			\Edge(r4)(r41);
			\Edge(r4)(r42);
			\Edge(r4)(r43);
			\draw (-0.5,-5.5) -- (-0.5,-4.5)[dashed];
			\draw (0.5,-6.5) -- (1.5,-6.5) [dashed];
			\draw (-1.5,-6.5) -- (-2.5,-6.5)[dashed];
			\draw (2,-8) -- (2,-7)[dashed];
			\draw (3,-9) -- (4,-9)[dashed];
			\draw (2,-10) -- (2,-11)[dashed];
			\draw (0.5,-11.5) -- (1.5,-11.5)[dashed];
			\draw (-0.5,-12.5) -- (-0.5,-13.5)[dashed];
			\draw (-1.5,-11.5) -- (-2.5,-11.5)[dashed];
			\draw (-3,-10) -- (-3,-11)[dashed];
			\draw (-4,-9) -- (-5,-9)[dashed];
			\draw (-3,-8) -- (-3,-7)[dashed];
			\node at (-0.5,-4) {(c)};
		\end{tikzpicture}
	\end{center}
	\captionof{figure}{Three examples of star-like graphs. The dashed lines stand for copies of $\mathbb{N}$. The graph in (b) is also called a star graph, while (c) is a "trimming" of a $4$-regular tree.}
	\label{sg_fig}
\end{figure}
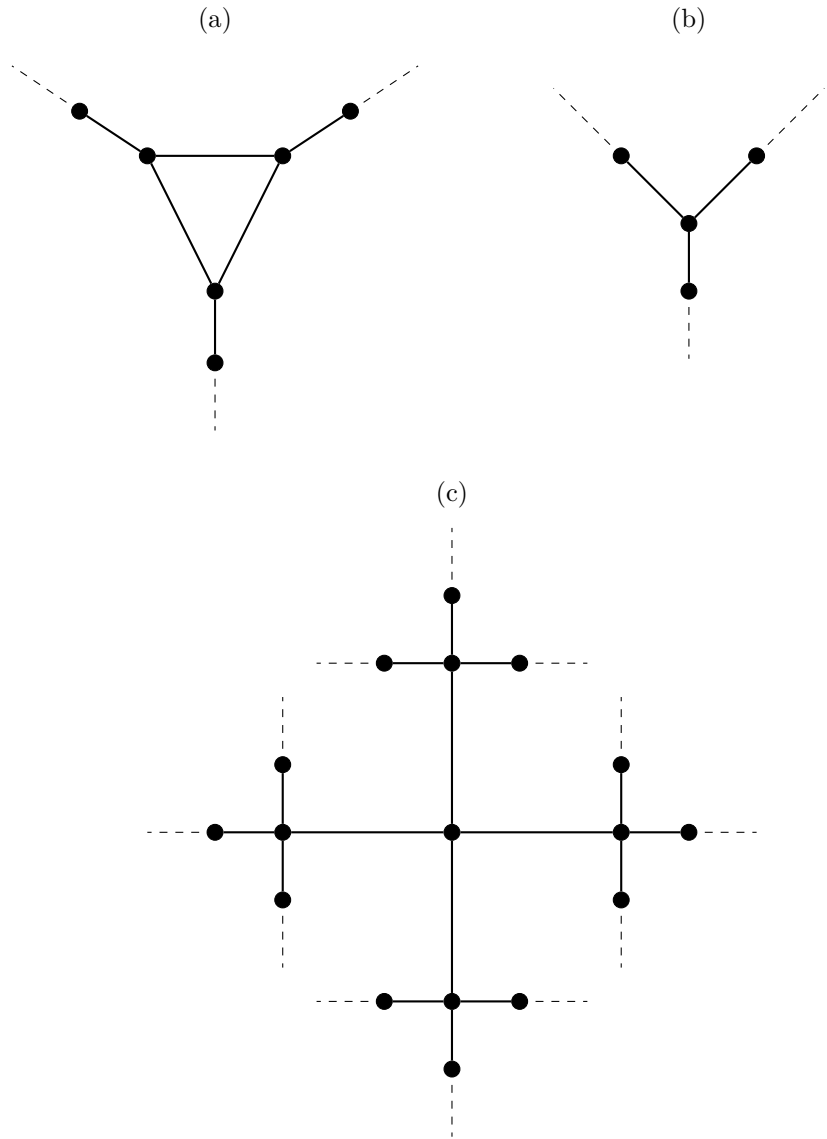
Our main result is the following
\begin{theorem}\label{main_thm}
	Let $G$ be a star-like graph, and let $J:\ell^2(G)\to\ell^2(G)$ be as in (\ref{op_eq}). Denote the spectral measure of $J$ by $P$, and let $P_s,P_{ac}$ be its singular and absolutely continuous parts $($with respect to the Lebesgue measure$)$ respectively. Then $P_s$ is supported on the set
	\begin{center}
		$S=\left\{E\in\mathbb{R}:\text{there exists a subordinate solution to (\ref{ev_eq}) on }G\right\}$
	\end{center}
	and $P_{ac}$ is supported on the set $N=\bigcup\limits_{v\in C}N_v$, where
	\begin{center}
		$N_v=\left\{E\in\mathbb{R}:\text{there exists no subordinate solution to (\ref{ev_eq}) on }\ell_v\right\}$.
	\end{center}
\end{theorem}

As noted before, when $G=\mathbb{N}$, $\delta_1$ is a cyclic vector for $J$. In particular, this implies that the spectrum is simple. This is no longer the case for when $G=\mathbb{Z}$, where it can be seen that the absolutely continuous spectrum may have multiplicity $2$. Nevertheless, Kac (\cite{Kac}) showed that the singular spectrum of Schr\"{o}dinger operators on the real line is simple. Later on, Gilbert (\cite{G1}) found a proof of this result using subordinacy theory, and Simon (\cite{Sim3}) found a proof of this fact using the theory of rank one perturbations. The local spectral multiplicity of $J$ can be described via a multiplicity function $N_J:\sigma(J)\to \mathbb{N}\cup\left\{\infty\right\}$ (see Section \ref{mul_sec} for a precise definition). The results of \cite{G1,Kac} essentially say that for $P_s$-almost every $E\in\mathbb{R}$, $N_J(E)=1$. A generalization of the above result in the continuous setting is given in \cite{SW}. They show that for a star-graph with $n$ branches, the local multiplicity is bounded by $n-1$, and give an explicit formula for $N_J$. Our second result is the following generalization: 
\begin{theorem}\label{mul_thm}
	Let $G,J,P$ be as in Theorem \ref{main_thm}. Given $E\in\mathbb{R}$, let
	\begin{center}
		$S(E)\coloneqq\left\{u:G\to\mathbb{R}:\text{u is a subordinate solution to (\ref{ev_eq}) on }G\right\}$.
	\end{center}
	Then, for $P_s$-almost every $E\in\mathbb{R}$, $N_J(E)\leq \dim S\left(E\right)$.
\end{theorem}
\begin{remark}
	It is not hard to show that on each half-line, if a subordinate solution exists then it is unique, which implies that $S\left(E\right)$ is a finite-dimensional space and so $\dim S\left(E\right)$ is well defined.
\end{remark}
When $G=\mathbb{Z}$, it can be seen that $\dim S(E)\leq 1$ for any $E\in\mathbb{R}$, and so simplicity of the singular spectrum in that case is given immediately by Theorem \ref{mul_thm}. With very little additional work, some of the results in \cite{SW} can also be obtained in the discrete setting. In general, the inequality cannot be replaced with an equality. We will present an example in which the inequality is strict. Note that by our definition, a solution which is only supported on the compact component is also subordinate and so in general, the multiplicity may exceed the number of half-lines attached to $C$. However, we will show that in the purely singular continuous part of the spectrum this is not possible. Specifically, we show that for $P_s$-almost every $E\in\sigma_s\left(J\right)\setminus\sigma_{pp}\left(J\right)$, $N_J\left(E\right)$ is bounded by $k$, where $k$ is the number of half-lines attached to $G$. We believe that the bound can be improved to $k-1$, as this is the case for star-graphs, as shown in \cite{SW} for the continuous setting, and in this work for the discrete one. However, our attempts to prove this bound did not succeed.

The rest of the paper is structured as follows. In Section 2, we give some basic measure-theoretic background, present the one-dimensional theory, and introduce the notion of star-like graphs. In Sections 3 and 4, we give proofs of theorems \ref{main_thm} and \ref{mul_thm} respectively. Section 5 includes some remarks, examples and applications.

{\bf Acknowledgments} I would like to thank my advisor Prof.\ Jonathan Breuer for his valuable guidance throughout this work. I would also like to thank Prof.\ Barry Simon for useful discussions.
\section{Preliminaries}
We begin by introducing relevant definitions and results regarding the boundary values of Borel transforms of measures.
\subsection{Boundary Behavior of Borel Transforms}
Throughout this work we deal with finite complex-valued Borel measures on $\mathbb{R}$. Given such a measure $\mu$, its Borel transform is defined by
\begin{center}
	$m_\mu(z)=\int_\mathbb{R}\frac{d\mu(x)}{x-z}$.
\end{center}
It is an analytic function defined on $\mathbb{C}_+\coloneqq\left\{z\in\mathbb{C}:\im{z}>0\right\}$. As we are interested in the boundary values of such functions, given any analytic function $F:\mathbb{C}_+\to\mathbb{C}_+$ and $E\in\mathbb{R}$, we denote $F(E+i0)\coloneqq\underset{\epsilon\rightarrow0}{\lim}F(E+i\epsilon)$ whenever the limit exists. In the case that $\mu$ is a positive measure which satisfies
\begin{equation}\label{measure_cond}
	\int_\mathbb{R}\frac{d\mu(x)}{|x|+1}<\infty,
\end{equation}
various continuity properties of $\mu$ with respect to the Lebesgue measure, which we denote throughout by $\lambda$, are related to the boundary values of its Borel transform. In particular, we will use the following well known theorem (see, for example, \cite{Sim})
\begin{theorem}\label{supp_thm}
	Let $\mu$ be a positive measure satisfying (\ref{measure_cond}). Denote by $\mu_{ac},\mu_s$ the absolutely continuous and singular parts of $\mu$ (with respect to the Lebesgue measure) respectively. Then
	\begin{enumerate}
		\item $\mu_{ac}$ is supported on the set $\left\{E\in\mathbb{R}:0<\im m_\mu(E+i0)<\infty\right\}$.
		
		\item $\mu_s$ is supported on the set $\left\{E\in\mathbb{R}:\im m_\mu(E+i0)=\infty\right\}$.
	\end{enumerate}
\end{theorem}
The second type of results that we will need concern the boundary behavior of ratios of Borel transforms. Specifically, given two Borel measures $\mu,\sigma$, let $\frac{d\mu}{d\sigma}(E)\coloneqq\underset{\epsilon\rightarrow0}{\lim}\frac{\mu(E-\epsilon,E+\epsilon)}{\sigma(E-\epsilon,E+\epsilon)}$ whenever it exists in $\mathbb{C}\cup\left\{\infty\right\}$. Note that if $\mu\ll\sigma$, then $\frac{d\mu}{d\sigma}$ coincides with the Radon-Nikodym derivative of $\mu$ w.r.t.\ $\sigma$ (as $L^1(\sigma)$ functions).  We will use the following version of Poltoratskii's theorem:
\begin{theorem}\emph{(\cite{JaL, Pol})}\label{polth}
	Let $\sigma,\mu$ be complex-valued Borel measures on $\mathbb{R}$ such that $\mu\ll\sigma$. Denote by $\sigma_s$ the part of $\sigma$ which is singular w.r.t.\ the Lebesgue measure. Then, for $\sigma_s$-almost every $E\in\mathbb{R}$, the limit $\underset{\epsilon\rightarrow0}{\lim}\frac{m_\mu(E+i\epsilon)}{m_\sigma(E+i\epsilon)}$ exists and is equal to $\frac{d\mu}{d\sigma}(E)$.
\end{theorem}
We will also need
\begin{theorem}\emph{(\cite{Kac})}\label{kac_thm}
	Let $\mu$ be a real measure and $\sigma$ be a probability measure, and let $E\in\mathbb{R}$. Assume that $\frac{d\mu}{d\sigma}(E)$ exists in $\mathbb{R}$, and that $\frac{d\sigma}{d\lambda}(E)$ exists, possibly equal to infinity. Then
	\begin{center}
		$\underset{\epsilon\rightarrow0}{\lim}\frac{\im(m_\mu(E+i\epsilon))}{\im(m_\sigma(E+i\epsilon))}=\frac{d\mu}{d\sigma}(E)$.
	\end{center}
\end{theorem}
If $\mu$ is a probability measure on $\mathbb{R}$, then as mentioned before, its Borel transform $m_\mu$ maps $\mathbb{C}_+\coloneqq\left\{z\in\mathbb{C}:\im z>0\right\}$ to itself analytically. Such functions are called Herglotz (sometimes Nevanlinna or Pick) functions. This correspondence also works in the other way (see e.g.\ \cite{Sim2}), in the following sense: if $m(z):\mathbb{C}_+\to\mathbb{C}_+$ is analytic and satisfies
\begin{center}
	$m(z)=z^{-1}+O(z^{-2})$,
\end{center}
then $m$ arises as the Borel transform of some probability measure $\mu$. In this work, we will use the following generalization of the connections mentioned above:
\begin{theorem}\emph{(\cite{GT})}\label{supp_thm_m}
	Let \mbox{$M:\mathbb{C}_+\to M_n(\mathbb{C})$} be an analytic matrix-valued function such that $\im M(z)$ is positive semi-definite for all $z\in\mathbb{C}_+$. Then
	\begin{enumerate}
		\item there exists a matrix-valued measure $\Omega$ such that
		\begin{center}
			$M(z)=C+Dz+\int_\mathbb{R}\frac{1+xz}{x-z}d\Omega(x)$
		\end{center}
		where $C$ is a self-adjoint matrix, and $D$ is positive definite.
		\item The singular part of $\Omega$ with respect to $\lambda$ is supported on the set $\left\{E\in\mathbb{R}:\im\tr M(E+i0)=\infty\right\}$.
	\end{enumerate}
\end{theorem}
\begin{remark}
	For $M\in M_n(\mathbb{C}), \im{A}=\frac{1}{2i}\left(M-M^*\right)$.
\end{remark}
\subsection{Subordinacy Theory in the Half Line Case}
Let $J$ be a Jacobi matrix on $\mathbb{N}$, and for $\theta\in\left[0,\pi\right)$ define the operator $J_\theta$ by $J_\theta=J-\tan\left(\theta\right)\langle\delta_1,\cdot\rangle\delta_1$. It can be verified that $\delta_1$ is a cyclic vector for $J_\theta$, namely
\begin{equation}\label{cyc_eq}
	\ell^2(\mathbb{N})=\overline{\vspan\left\{J_\theta^k\delta_1:k\in\mathbb{N}\cup\left\{0\right\}\right\}}.
\end{equation}
Recall that given a self-adjoint operator $H$ defined on a Hilbert space $\mathcal{H}$ and $\psi,\varphi\in\mathcal{H}$, the spectral measure of $\psi$ and $\varphi$ w.r.t.\ $H$ is defined to be the unique Borel measure which satisfies
\begin{center}
	$\langle \psi,f\left(H\right)\varphi\rangle=\int_{\sigma\left(H\right)} f(t)d\mu_{\psi,\varphi}\left(t\right)$
\end{center}
for any continuous function $f:\sigma\left(H\right)\to\mathbb{C}$. If $\psi=\varphi$, we refer to this measure as the spectral measure of $\psi$ w.r.t.\ $H$. Denote by $\mu_\theta$ the spectral measure of $\delta_1$ w.r.t.\ $J_\theta$. It is not hard to see that (\ref{cyc_eq}) implies that for any $v\in\ell^2\left(N\right)$, the spectral measure of $v$ is absolutely continuous w.r.t.\ $\mu_\theta$. Thus, our task is to study of continuity properties of $\mu_\theta$. As noted before, these properties are related to its Borel transform which we denote by $m_\theta$. Note that by the definitions of $\mu_\theta$ and $m_\theta$, we have
\begin{equation}\label{m_hl_eq}
	m_\theta(z)=\int_\mathbb{R}\frac{d\mu_\theta(x)}{x-z}=\langle\delta_1,\left(J_\theta-z\right)^{-1}\delta_1\rangle.
\end{equation}
The main result in half-line subordinacy theory is the following.
\begin{theorem}\label{gilpe_th}
	For any $E\in\mathbb{R}$ and $\theta\in\left[0,\pi\right)$, $\im m_\theta\left(E+i0\right)=\infty$ if and only if $u_{\theta,E}$ is subordinate.
\end{theorem}
Theorem \ref{gilpe_th} was originally proved in the continuous setting by Gilbert and Pearson in \cite{GP}. Its discrete analogue is a direct consequence of the results presented in \cite{JL1}.
\begin{remark}
	In particular, combining Theorems \ref{supp_thm} and \ref{gilpe_th}, we get that the singular part of $\mu_\theta$ is supported on the set of energies for which $u_{\theta,E}$ is subordinate.
\end{remark}

\subsection{Jacobi Matrices on Star-like Graphs}\label{per_slike_subsection}
We begin by introducing the notion of a {\it star-like} graph. Let $C=\langle V_C,E_C\rangle$ be a finite connected graph. A star-like graph is given by selecting a subset of vertices $V_0\subseteq V_C$, and attaching a copy of $\mathbb{N}$ to each vertex $v\in V_0$. Formally, $G=\langle V,E\rangle$ is given by
\begin{center}
	$V=V_C\sqcup\left\{v_i^u\right\}_{i\in\mathbb{N},u\in V_0}$\\
	$E=E_C\sqcup\left\{\left(u,v_1^u\right):i\in V_0\right\}\sqcup\left\{\left(v_i^u,v_{i+1}^u\right):i\in\mathbb{N},u\in V_0\right\}$
\end{center}
where $\sqcup$ denotes a disjoint union.
\begin{remark}
	Given a star-like graph $G$, it is possible to construct $G$ by choosing different finite sub-graphs $C_1,C_2$ of it and applying the above procedure. However, our results do not depend in the choice of $C$, and so we fix such a finite graph $C$ and from now on we refer to it as {\it the compact component} of $G$.
\end{remark}
Let $J$ be a Jacobi matrix on $G$ which acts on $\ell^2\left(G\right)$ as in (\ref{op_eq}). Unlike the half-line case, there need not be a cyclic vector for $J$. However, we have the following:
\begin{claim}\label{cyclic_clame}
	$\ell^2(G)=\overline{\vspan\left\{J^k\delta_v:v\in V_C,k\in\mathbb{N}\cup\left\{0\right\}\right\}}$.
\end{claim}
Namely, the set $\left\{\delta_v:v\in V_C\right\}$ is cyclic for $J$.
\begin{proof}
	We give a sketch of the proof. It suffices to show that
	\begin{center}
		$\left\{\delta_v:v\in V\right\}\subseteq\vspan{\left\{J^k\delta_v:v\in V_C,k\in\mathbb{N}\cup\left\{0\right\}\right\}}\coloneqq A$.
	\end{center}
	Let $v\in V$. If $v\in V_C$, then clearly $\delta_v\in A$. Otherwise, denote by $d(v,C)$ the length of a minimal path from $v$ to some vertex in $V_C$, and let \mbox{$B_k=\left\{v\in V:d(v,C)=k\right\}$}. If $v\in B_1$, then by the definition of $G$, there exists some unique $u\in V_C$ such that $v\sim u$. By the definition of $J$,
	\begin{center}
		$J\delta_u=b_u\delta_u+a_{\left(u,v\right)}\delta_v+\sum\limits_{u\sim w\in C}a_{\left(u,w\right)}\delta_w$.
	\end{center}
	Now, since $J\delta_u,\,b_u\delta_u+\sum\limits_{u\sim w\in C}a_{\left(u,w\right)}\delta_w\in A$, we get that $a_{\left(u,v\right)}\delta_v\in A$ as the difference of elements in $A$, and since $a_{\left(u,v\right)}\neq 0$, we get that $\delta_v\in A$. Thus, we get that $B_1\subseteq A$. Now, noting that by connectedness of $G$ and by the fact that any connected component of the induced graph on $V\setminus V_C$ is isomorphic to $\mathbb{N}$, we get that every $v\in B_2$ has a unique vertex in $u\in B_1$ such that $v\sim u$. From here, one can proceed by induction.
\end{proof}
Denote $V_C=\left\{v_1,\ldots,v_n\right\}$, and let $\delta_k\coloneqq\delta_{v_k}$. Claim \ref{cyclic_clame} implies that every spectral measure of $J$ is absolutely continuous with respect to the measure $\mu\coloneqq\sum\limits_{k=1}^{n} \mu_k$ where $\mu_k$ is the spectral measure of $\delta_k$, and so our purpose is to study the continuity properties of $\mu$. Given $z\in\mathbb{C}_+$, define $M(z)\in M_n(\mathbb{C})$ by
\begin{equation}\label{Mz_def}
	\left(M(z)\right)_{kj}=\langle\delta_k,\left(J-z\right)^{-1}\delta_j\rangle.
\end{equation}
$M:\mathbb{C}_+\to M_n(\mathbb{C})$ is analytic and for any $z\in\mathbb{C}_+$, $\im{M(z)}\geq 0$. Let $\Omega$ be the matrix-valued measure given by Theorem \ref{supp_thm_m}. By the definition of the spectral measure we get that for any $1\leq k\leq n$, $\mu_k=\Omega_{kk}$, and so $\mu=\tr{\Omega}$. For $1\leq k\leq n$, define $G_k$ in the following way:
\begin{itemize}
	\item If $v_k\in V_0$, then $G_k$ is the graph induced on $\left\{v_k\right\}\cup \left\{v_i^{v_k}:i\in\mathbb{N}\right\}$. Note that $G_k\cong\mathbb{N}$ with $v_k$ as the origin.
	\item Otherwise, $G_k$ consists of the singleton $\left\{v_k\right\}$.
\end{itemize}
Finally, denote by $J_k$ the operator $P_kJP_k$ on $\ell^2(G_k)$, where $P_k$ is the projection of $\ell^2(G)$ onto $\ell^2(G_k)$. Note that in the first case, $J_k$ is a Jacobi matrix on the half-line $G_k$, and in the latter case $J_k$ acts on $\mathbb{C}$ by multiplication by $b_{v_k}$. In any case, $J_k$ is a self-adjoint operator.

\section{Proof of Theorem \ref{main_thm}}
Let $1\leq k\leq n$ and let $z\in\mathbb{C}_+$. Denote by $\widetilde{u}_k\in\ell^2(G_k)$ the unique $\ell^2$ solution to $\left(J_k-z\right)u=\delta_k$, and by $\varphi_k\in\ell^2(G)$ the unique $\ell^2$ solution to $\left(J-z\right)\varphi=\delta_k$ (these solutions exist and are unique since $J_k,J$ are self-adjoint). Let $u_k\in\ell^2(G)$ be defined by
\begin{center}
	$\left(u_k\right)_{v}=\begin{cases}
		\left(\widetilde{u}_k\right)_{v} & w\in G_k\\
		0 & \text{otherwise}
	\end{cases}$.
\end{center}
For $1\leq j<k\leq n$, the supports of $u_j$ and $u_k$ are disjoint, and so the collection $\left\{u_k:1\leq k\leq n\right\}$ is linearly independent. Denote by $m_k$ the Borel transform of the spectral measure of $\delta_k$ with respect to $J_k$. By (\ref{m_hl_eq}) $m_k(z)=\left(u_k\right)_1$, and so we have
\begin{center}
	$\left((J-z)u_k\right)_v=\begin{cases}
		1 & v=v_k\\
		a_{\left(v_k,v\right)}m_k(z) & v\sim v_k,v\in C\\
		0 & \text{otherwise}
	\end{cases}$
\end{center}
For $1\leq k\leq n$, define $w_k\in\mathbb{C}^n$ by $w_{k}^{j}=\left(\left(J-z\right)u_k\right)_{v_j}$.
\begin{claim}\label{indep_claim}
	The collection $\left\{w_k:1\leq k\leq n\right\}$ is linearly independent.
\end{claim}
\begin{proof}
	Assume by contradiction that $\sum\limits_{k=1}^{n}\alpha_kw_k=0$, and there exists $1\leq j\leq n$ such that $\alpha_j\neq 0$. Let $u=\sum\limits_{k=1}^{n}\alpha_ku_k$. The fact that $\alpha_j\neq 0$ along with the disjointness of the supports of $u_j$ and $u_k$ for any $k\neq j$ implies that $u|_{G_j}\neq 0$. Now, note that for any $1\leq k\leq n$, $\left(J-z\right)u_k$ vanishes outside $V(C)$ which implies that $\left(J-z\right)u$ also vanishes outside $V(C)$. In addition, $\left(\left(J-z\right)u_k\right)|_{V(C)}=w_k$, and so $\left(\left(J-z\right)u\right)|_{V(C)}=\sum\limits_{k=1}^{n}\alpha_kw_k=0$. This implies that $z$ is an eigenvalue of $J$, which is a contradiction to $J$ being self-adjoint.
\end{proof}
The following lemma is crucial to the proof of Theorem \ref{main_thm}.
\begin{lemma}\label{Schur_lemma}
	Let $M(z)$ be as in (\ref{Mz_def}). Then for any $z\in\mathbb{C}_+$, $M(z)$ is invertible and
	\begin{equation}\label{adj_eq}
		M(z)^{-1}=A+\diag(\frac{1}{m_1(z)},\ldots,\frac{1}{m_n(z)})
	\end{equation}
	where $A$ is defined by
	\begin{center}
		$A_{ij}=\begin{cases}
			a_{\left(v_i,v_j\right)} & v_i\sim v_j\\
			0 & otherwise
		\end{cases}$
	\end{center}
\end{lemma}
\begin{proof}
	Let $1\leq k\leq n$. Note that if $\alpha_1,\ldots,\alpha_n\in\mathbb{C}$ satisfy
	\begin{center}
		$\sum\limits_{j=1}^{n}\alpha_jw_j=\delta_k$,
	\end{center}
	then $\varphi_k=\sum\limits_{j=1}^{n}\alpha_ju_j$. Denote by $F(z)$ the matrix whose $j$'th column is $w_j$. By claim \ref{indep_claim}, $F(z)$ is invertible for every $z\in\mathbb{C}_+$, and $\alpha_1,\ldots,\alpha_n$ can be retrieved by the equation
	\begin{center}
		$F(z)\left(\begin{matrix}
			x_1 \\
			\vdots \\
			x_n
		\end{matrix}\right)=e_k$,
	\end{center}
	i.e.
	\begin{center}
		$\left(\begin{matrix}
			\alpha_1 \\
			\vdots \\
			\alpha_n
		\end{matrix}\right)=F(z)^{-1}e_k$,
	\end{center}
	where $e_k$ is the $k$'th element in the standard basis of $\mathbb{C}^n$. Thus, we have that $\alpha_j=\left(F(z)\right)^{-1}_{jk}$. Note that
	\begin{center}
		$\left(M(z)\right)_{kj}=\langle\delta_k,\left(J-z\right)^{-1}\delta_j\rangle=\left(\varphi_j\right)_{v_k}$
	\end{center}
	and since $u_k(v_j)=0$ whenever $k\neq j$ and $\left(u_k\right)_{v_k}=m_k(z)$, we get that
	\begin{center}
		$\left(M(z)\right)_{kj}=\left(F(z)\right)^{-1}_{jk}\cdot m_j(z)$,
	\end{center}
	i.e.
	\begin{center}
		$M(z)=\diag\left(m_1(z),\ldots,m_n(z)\right)\cdot \left(F(z)\right)^{-1}$.
	\end{center}
	Since for any $1\leq k\leq n$, $m_k(z)\in\mathbb{C}_+$, and since $F(z)$ is invertible, we get that $M(z)$ is invertible and
		\begin{center}
			$M(z)^{-1}=F(z)\cdot\diag\left(\frac{1}{m_1(z)},\ldots,\frac{1}{m_n(z)}\right)$.
		\end{center}
	Now, (\ref{adj_eq}) follows from a straightforward calculation.
\end{proof}
\begin{remark}
	This is a special case of an identity known as the Krein, Feshbach or Schur formula (see, e.g., \cite[Chapter 5]{AW}), which says that given a self-adjoint operator $H$ acting on a Hilbert space $\mathcal{H}$ and $P$ the orthogonal projection onto a finite-dimensional subspace $\mathcal{H}_0\subseteq\mathcal{H}$, if we write
	\begin{center}
		$H=PHP+\hat{H}$,
	\end{center}
	then, denoting $H_P\coloneqq PHP$, the following identity holds:
	\begin{equation}\label{sch_id}
		P\left(H-z\right)^{-1}P=\left[H_P+\left[P\left(\hat{H}-z\right)^{-1}P\right]^{-1}_P\right]^{-1}_P
	\end{equation}
	where $\left[\cdot\right]^{-1}_P$ indicates that we take the inverse only on $\mathcal{H}_0$. Note that in our case, $M\left(z\right)=P_C\left(J-z\right)^{-1}P_C$, and so (\ref{adj_eq}) is given by taking the inverse (inside $\mathcal{H}_0$) on both sides of (\ref{sch_id}), and explicitly computing the RHS. The proof of Lemma \ref{Schur_lemma} is a more direct way of obtaining this result.
\end{remark}
We will also need the following
\begin{claim}\label{ext_claim}
	Let $E\in\mathbb{R}$ and denote $J|_C=P_CJP_C$, where $P_C$ is the projection of $\ell^2\left(G\right)$ onto $\ell^2\left(C\right)$. Suppose that $\widetilde{\varphi}\in\ell^2\left(C\right)$ satisfies
	\begin{center}
		$\left(J|_C\widetilde{\varphi}\right)_{v_k}=E\varphi_{v_k}$
	\end{center}
	for any $1\leq k\leq n$ such that $G_k=\left\{v_k\right\}$.
	Then $\widetilde{\varphi}$ can be extended to a function $\varphi:G\to\mathbb{C}$ such that $J\varphi=E\varphi$.
\end{claim}
\begin{remark}
	Note that $\varphi$ need not be in $\ell^2(G)$, and so the expression $J\varphi=E\varphi$ should be thought of as a difference equation.
\end{remark}
\begin{proof}
	First, note that whenever $G_k$ consists of a single vertex, $v_k$ has no neighbors outside $C$ and so for any $\psi:G\to\mathbb{C}$ which extends $\widetilde{\varphi}$,
	\begin{center}
		$J|_C\widetilde{\varphi}_{v_k}=\left(J\psi\right)_{v_k}=E\psi_{v_k}$.
	\end{center}
	For any $1\leq k\leq n$ such that $G_k\cong\mathbb{N}$, the values of $\varphi$ on $G_k$ can be determined via the difference equation.
\end{proof}
 We are now ready to prove Theorem \ref{main_thm}.
\begin{proof}
	 Using the notations from the previous section, for $1\leq j,k\leq n$ denote by $\mu_{jk}$ the spectral measure of $\delta_j$ and $\delta_k$, i.e.\ $\mu_{jk}$ is the unique Borel measure that satisfies
	 \begin{center}
	 	$\langle\delta_j,f(H)\delta_k\rangle=\int f(x)d\mu_{jk}(x)$
	 \end{center}
 	for any continuous function on the spectrum of $J$. Also, denote $\mu_{kk}\coloneqq \mu_k$, and $\mu\coloneqq \sum\limits_{k=1}^{n}\mu_k$. Note that
 	\begin{center}
 		$\left(M(z)\right)_{jk}=\int\frac{d\mu_{jk}(x)}{x-z}$
 	\end{center}
 	i.e.\ $\left(M(z)\right)_{jk}$ is the Borel transform of $\mu_{jk}$, and $\tr{M}$ is the Borel transform of $\mu$. For $1\leq k\leq n$, let
	\begin{center}
		$A_k=\left\{E\in\mathbb{R}:\frac{d\mu_k}{d\mu}(E)>0\right\}$.
	\end{center}
	Since $1=\frac{d\mu}{d\mu}=\sum\limits_{k=1}^{n}\frac{d\mu_k}{d\mu}$, we get that $A\coloneqq\bigcup\limits_{k=1}^{n}A_k$ supports $\mu$. Let $1\leq k\leq n$ and let $E\in A_k$. By Theorem \ref{supp_thm}, we may assume that $\underset{\epsilon\rightarrow0}{\lim}\im(\tr{M}(E+i\epsilon))=\infty$. By Theorem \ref{kac_thm}, we get that
	\begin{center}
		$\underset{\epsilon\rightarrow0}{\lim}\frac{\im(M_{kk}(E+i\epsilon))}{\im(\tr{M}(E+i\epsilon))}>0$
	\end{center}
	and so $\underset{\epsilon\rightarrow0}{\lim}\im(M_{kk}(E+i\epsilon))=\infty$. Thus, we may assume that $\mu_k|_{A_k}$ is singular with respect to the Lebesgue measure. In addition, $\mu|_{A_k}\ll\mu_k|_{A_k}$, and so for any $1\leq j\leq n$, $\mu_{jk}|_{A_k}\ll\mu_k|_{A_k}$. By Theorem \ref{polth}, this implies that for $\mu_k$-almost every $E\in A_k$ the limit
	\begin{center}
		$\underset{\epsilon\rightarrow0}{\lim}\frac{M_{jk}(E+i\epsilon)}{M_{kk}(E+i\epsilon)}\coloneqq\alpha_j$
	\end{center}
	Now, let $e_k$ be the $k$'th element in the standard basis of $\mathbb{C}^n$. We have
	\begin{center}
		$0=\underset{\epsilon\rightarrow0}{\lim}\frac{1}{|M_{kk}(E+i\epsilon)|}\|e_k\|=
		\underset{\epsilon\rightarrow0}{\lim}\|\frac{e_k}{M_{kk}(E+i\epsilon)}\|=
		\underset{\epsilon\rightarrow0}{\lim}\|M\left(E+i\epsilon\right)^{-1}\left(M(E+i\epsilon)\frac{e_k}{M_{kk}(E+i\epsilon)}\right)\|$.
	\end{center}
	By the fact that $M(E+i\epsilon)(e_k)=\begin{pmatrix}
		M_{1k}(E+i\epsilon)\\
		\vdots\\
		M_{nk}(E+i\epsilon)
	\end{pmatrix}$, we get
	\begin{center}
		$0=\underset{\epsilon\rightarrow0}{\lim}M(E+i\epsilon)^{-1}\begin{pmatrix}
			\frac{M_{1k}(E+i\epsilon)}{M_{kk}(E+i\epsilon)}\\
			\vdots\\
			\frac{M_{nk}(E+i\epsilon)}{M_{kk}(E+i\epsilon)}
		\end{pmatrix}$.
	\end{center}
	Taking into account (\ref{adj_eq}), we get that for every $1\leq j\leq n$,
	\begin{equation}\label{limit_eq}
		\underset{\epsilon\rightarrow0}{\lim}\sum\limits_{\underset{v_l\sim v_j}{1\leq l\leq n}}a_{\left(v_l,v_j\right)}\frac{M_{lk}(E+i\epsilon)}{M_{kk}(E+i\epsilon)}+\frac{M_{jk}(E+i\epsilon)}{M_{kk}(E+i\epsilon)}\cdot\frac{1}{m_j(E+i\epsilon)}=0.
	\end{equation}
	Define $\widetilde{\varphi}$ on $C$ by $\widetilde{\varphi}_{v_j}=\alpha_j$ for every $1\leq j\leq n$. We claim that $\left(J|_C\widetilde{\varphi}\right)_{v_j}=E\widetilde{\varphi}_{v_j}$ whenever $G_j$ consists of a single vertex, and so by Claim \ref{ext_claim}, it can be extended to a solution $\varphi$ to (\ref{ev_eq}) on $G$. So fix such a $j$. Clearly, we have
	\begin{center}
		$\underset{\epsilon\rightarrow0}{\lim}\frac{1}{m_j(E+i\epsilon)}=q_j-E$
	\end{center}
	and by (\ref{limit_eq}), we get that $\sum\limits_{\underset{v_l\sim v_j}{1\leq l\leq n}}a_{\left(v_l,v_j\right)}\alpha_l=-(b_j-E)\alpha_j$, and so
	\begin{center}
		$\left(J|_C\widetilde{\varphi}\right)_{v_j}=\sum\limits_{\underset{v_l\sim v_j}{1\leq l\leq n}}a_{\left(v_l,v_j\right)}\alpha_l+b_j\alpha_j=-(q_j-E)\alpha_j+b_j\alpha_j=E\alpha_j=E\widetilde{\varphi}_{v_j}$
	\end{center}
	as required. Let $\varphi$ be the completion of $\widetilde{\varphi}$ to a solution of (\ref{ev_eq}) on $G$. Note that $\varphi$ is non-trivial since $\varphi_{v_k}=\alpha_k=1$. We claim that $\varphi$ is subordinate. Let $1\leq j\leq n$ such that $G_j$ is a half-line.
	
	If $\alpha_j=0$ and there exists a sequence $\left(\epsilon_n\right)_{n=1}^{\infty}$ such that $\underset{n\rightarrow\infty}{\lim}m_j(E+i\epsilon_n)\neq 0$, then we get that $\sum\limits_{\underset{v_l\sim v_j}{1\leq l\leq n}}a_{\left(v_l,v_j\right)}\alpha_l=0$ and so by the difference equation which $\varphi$ satisfies, we get that $\varphi$ vanishes on $G_j$. Otherwise, $\underset{\epsilon\rightarrow0}{\lim}m_j(E+i\epsilon)=0$, and by Theorem \ref{gilpe_th}, any solution which satisfies $\varphi(v_j)=0$ is subordinate on $G_j$ (since in that case, $\lim\limits_{\epsilon\rightarrow0}m_j(E+i\epsilon)=\cot\theta$ for $\theta=\frac{\pi}{2}$).
	
	Assume now that $\alpha_j\neq0$. By (\ref{limit_eq}), this implies that the limit $\underset{\epsilon\rightarrow0}{\lim}\frac{1}{m_j(E+i\epsilon)}$ exists and is real, and so there exists $\theta\in\left[0,\pi\right)$ such that $\underset{\epsilon\rightarrow0}{\lim}m(E+i\epsilon)=\cot(\theta)$. By Theorem \ref{gilpe_th}, this implies that if there exists $\lambda\neq 0$ such that
	\begin{equation}\label{bc_sub_sol}
		\left(\sum\limits_{\underset{v_l\sim v_j}{1\leq l\leq n}}a_{\left(v_l,v_j\right)}\alpha_l,\alpha_j\right)=\lambda\left(-\sin(\theta),\cos(\theta)\right)
	\end{equation}
	then $\varphi$ is subordinate on $G_j$. Note that $\alpha_j\neq0$ implies that $\theta\neq\frac{\pi}{2}$. By (\ref{limit_eq}), we get that
	\begin{center}
		$\sum\limits_{\underset{v_l\sim v_j}{1\leq l\leq n}}a_{\left(v_l,v_j\right)}\alpha_l=-\frac{\sin(\theta)}{\cos(\theta)}\alpha_j$
	\end{center}
	and so
	\begin{center}
		$\left(\sum\limits_{\underset{v_l\sim v_j}{1\leq l\leq n}}a_{\left(v_l,v_j\right)}\alpha_l,\alpha_j\right)=\left(-\frac{\sin(\theta)}{\cos(\theta)}\alpha_j,\alpha_j\right)$.
	\end{center}
	Multiplying by $\frac{\cos(\theta)}{\alpha_j}$, we get (\ref{bc_sub_sol}) as required.
\end{proof}
\begin{remark}\label{help_rem}
	Note that the choice of the k'th column, i.e.\ setting \mbox{$\alpha_j=\lim\limits_{\epsilon\rightarrow0}\frac{M_{jk}(E+i\epsilon)}{M_{kk}(E+i\epsilon)}$} is only done so that the resulting solution will be non-trivial. Namely, for any $1\leq l\leq n$, if there exists some $1\leq p\leq n$ such that $\underset{\epsilon\rightarrow0}{\lim}\frac{M_{pl}(E+i\epsilon)}{M_{kk}(E+i\epsilon)}\neq 0$, then setting
	$\alpha_j=\underset{\epsilon\rightarrow0}{\lim}\frac{M_{jl}(E+i\epsilon)}{M_{kk}(E+i\epsilon)}$ for any $1\leq j\leq n$ will also generate a subordinate solution.
\end{remark}
\section{Proof of Theorem \ref{mul_thm}}\label{mul_sec}
Let $H$ be a self-adjoint operator acting on a Hilbert space $\mathcal{H}$ and let $P$ be the projection-valued measure associated with $H$ by the spectral theorem. It is well known (see, e.g., \cite[Chapter 7]{BS}), that there exist a collection of Hilbert spaces $\left\{\mathcal{H}_E:E\in\mathbb{R}\right\}$ so that $H$ is unitarily equivalent to multiplication by the free variable $E$ on the space $\widetilde{\mathcal{H}}\coloneqq\int_{\mathbb{R}}^{\oplus}\mathcal{H}_Ed\mu\left(E\right)$ whenever $\mu$ is a Borel measure on $\sigma\left(H\right)$ for which $\mu\left(A\right)=0\iff P\left(A\right)=0$ for any Borel set $A$. The measure $\mu$ is not unique, but it is determined (up to unitary equivalence) by its null sets. In particular, if $v_1,\ldots,v_n$ is cyclic for $H$, then $\mu$ can be taken to be the sum $\mu_{v_1}+\ldots+\mu_{v_n}$. The spectral multiplicity function is then given by $N_H\left(E\right)=\dim\mathcal{H}_E$ and is defined $\mu$-almost everywhere.

Let $J$ be a Jacobi matrix on a star-like graph $G$, and let $M,\Omega,\mu$ be as in the previous section. We will use the following
\begin{prop}
	For $\mu_s$-almost every $E\in\mathbb{R}$,
	\begin{equation}\label{mult_eq}
		N_J\left(E\right)=\rank\omega\left(E\right)
	\end{equation}
	where $\omega\left(E\right)$ is defined by $\left(\omega\left(E\right)\right)_{ij}=\frac{d\mu_{ij}}{d\mu}\left(E\right)$.
\end{prop}
\begin{remark}
	As already mentioned, $\omega\left(E\right)$ is also defined $\mu_s$-almost everywhere.
\end{remark}
\begin{proof}
	Let $U:\ell^2\left(G\right)\to\widetilde{\mathcal{H}}$ be the unitary transformation which sastisfies $UJ\varphi=\widetilde{J}U\varphi$ for any $\varphi\in\ell^2\left(G\right)$, where $\widetilde{J}$ is the operator of multiplication by the free variable. Note that for any $\varphi\in\ell^2\left(G\right)$, $U\varphi$ is a vector-valued function $E\to U\varphi\left(E\right)$ such that $U\varphi\left(E\right)\in\mathcal{H}_E$ for $\mu$-almost every $E\in\mathbb{R}$, and $\int_\mathbb{R}\|U\varphi\left(E\right)\|^2d\mu\left(E\right)<\infty$. Recall that $\delta_1,\ldots,\delta_n$ is a cyclic set for $J$, and denote $\psi_j=U\delta_j$ for $1\leq j\leq n$. Let $\mu_{\psi_i,\psi_j}$ be the spectral measure of $\psi_i$ and $\psi_j$ w.r.t.\ $\widetilde{J}$. For any Borel set $A\subseteq\mathbb{R}$, we have
	\begin{center}
		$\mu_{\psi_i,\psi_j}\left(A\right)=\int\mathbbm{1}_Ad\mu_{\psi_i,\psi_j}=\langle\psi_i,\mathbbm{1}_A\left(
		\widetilde{J}\right)\psi_j\rangle=\int_A\langle\psi_i\left(E\right),\psi_j\left(E\right)\rangle_{\mathcal{H}_E}d\mu\left(E\right)$.
	\end{center}
	Thus, if we denote $f_{ij}\left(E\right)=\langle\psi_i\left(E\right),\psi_j\left(E\right)\rangle_{\mathcal{H}_E}$, then
	\begin{center}
		$d\mu_{ij}=d\mu_{\psi_i,\psi_j}=f_{ij}d\mu$.
	\end{center}
	and so we have
	\begin{center}
		$\left(\omega\left(E\right)\right)_{ij}=\frac{d\mu_{ij}}{d\mu}\left(E\right)=\langle\psi_i\left(E\right),\psi_j\left(E\right)\rangle_{\mathcal{H}_E}$
	\end{center}
	Now, we claim that for $\mu$-almost every $E\in\mathbb{R}$, $\vspan\left\{\psi_1\left(E\right),\ldots,\psi_n\left(E\right)\right\}=\mathcal{H}_E$. Assume not. Then there exists $0\neq\psi\in\int^{\oplus}_{\mathbb{R}}$ and a Borel set $A\subseteq\mathbb{R}$ such that $\mu\left(A\right)>0$ and $\psi\left(E\right)\perp\vspan\left\{\psi_1\left(E\right),\ldots,\psi_n\left(E\right)\right\}$ for $\mu$-almost every $E\in A$. Thus, for every $k\in\mathbb{N}$ and for every $1\leq i\leq n$, we have
	\begin{center}
			$0=\langle\widetilde{J}^k\psi_j,\psi\rangle_1=\langle\widetilde{J}^kUU^{-1}\psi_j,UU^{-1}\psi\rangle_1=$ \\$=\langle UJ^kU^{-1}\psi_j,UU^{-1}\psi\rangle_1=\langle J^k\delta_j,U^{-1}\psi\rangle_2$
	\end{center}
	where $\langle\cdot,\cdot\rangle_1$ is the inner product in the direct integral, and $\langle\cdot,\cdot\rangle_2$ is the inner product in $\ell^2\left(G\right)$. This implies that $U^{-1}\psi=0$ and so $\psi\equiv 0$ which contradicts our assumption. Finally, (\ref{mult_eq}) follows from the fact that if $\vspan\left\{v_1,...,v_k\right\}=\mathbb{C}^n$, then the rank of the matrix $A_{ij}=\langle v_i,v_j\rangle$ is $n$
\end{proof}
We are now ready to prove Theorem \ref{mul_thm}.
\begin{proof}
	As before, let $A_k=\left\{E\in\mathbb{R}:\frac{d\mu_k}{d\mu}(E)>0\right\}$ for $1\leq k\leq n$. By Theorem \ref{polth}, for every $1\leq j,l\leq n$ and for $\mu_s$-almost every $E\in A_k$, we have
	\begin{center}
		$\left(\omega(E)\right)_{lj}=\frac{d\mu_k}{d\mu}(E)\cdot\lim\limits_{\epsilon\rightarrow0}\frac{M_{lj}(E+i\epsilon)}{M_{kk}(E+i\epsilon)}$.
	\end{center}
	Since $0<\frac{d\mu_k}{d\mu}(E)<\infty$ for $\mu$-almost every $E\in A_k$, we get that if $\rank\omega(E)=m$, then, by Remark \ref{help_rem}, there are at least $m$ independent subordinate solutions to (\ref{ev_eq}), as required.
\end{proof}
In the next section, we will construct an example in which the inequality in Theorem \ref{mul_thm} is strict. Also, in the case where $G=\mathbb{Z}$, the inequality becomes an equality, and so in that sense, this result is optimal. However, the next proposition shows that in the purely singular continuous part of the spectrum, Theorem \ref{mul_thm} can be improved.
\begin{prop}\label{scp_prop}
	Denote by $k$ the number of half-lines emanating from $C$, i.e.\ $k=\#\left\{j:G_j\cong\mathbb{N}\right\}$ and assume that $k>1$. Then for $\mu$-almost every \mbox{$E\in\sigma_{sc}\left(H\right)\setminus\sigma_{pp}\left(H\right)$}, \mbox{$N_H\left(E\right)\leq k$}.
\end{prop}
\begin{proof}
	Assume not. Let $E\in\sigma_{sc}\left(H\right)\setminus\sigma_{pp}\left(H\right)$ and let $\varphi_1,\ldots,\varphi_{k+1}$ be linearly independent solutions of $H\varphi=E\varphi$. For every $1\leq j\leq k$, let $v_j\in G_j$ be a vertex on which every non-zero subordinate solution on $G_j$ does not vanish. Such a vertex exists due to the uniqueness of the subordinate solution on a half-line. For any $1\leq i\leq k+1$, define $u_i=\left(\varphi_i\left(v_1\right),\ldots,\varphi_i\left(v_k\right)\right)$. Then $\left\{u_1,\ldots,u_{k+1}\right\}$ is a subset of $\mathbb{C}^k$ with $k+1$ vectors, and so it is linearly dependent. Then there exist $\alpha_1,\ldots,\alpha_{k+1}\in\mathbb{C}$ s.t.\ $\sum\limits_{i=1}^{k+1}\alpha_iu_i=0$. By the uniqueness of subordinate solutions on half-lines, this implies that the solution $\varphi=\sum\limits_{i=1}^{k+1}\alpha_i\varphi_i$ vanishes on each half-line. But $\varphi\neq 0$ since the set $\varphi_1,\ldots,\varphi_{k+1}$ is linearly independent, and so $\varphi$ is supported on a finite set, and in particular $\varphi\in\ell^2\left(G\right)$. This implies that $E\in\sigma_{pp}\left(H\right)$, which contradicts our assumption.
\end{proof}

\section{Remarks}
\subsection{The multiplicity of Schr\"{o}dinger operators on star-graphs}
In \cite{SW}, the multiplicity of Schr\"{o}dinger operators on star-graphs is studied in the continuous setting. The graph $\Gamma$ is given by the gluing of a finite number of half-lines, and a Schr\"{o}dinger operator $H$ on $\Gamma$ is given in the following way. On each half-line $\ell$, $H$ acts as a Schr\"{o}dinger operator on $\ell$, i.e.\ there exists $q_\ell:\mathbb{R}_{\geq 0}\to\mathbb{R}$ s.t.\ for every $\varphi$ in the domain of $H$,
\begin{equation}\label{eq_loc_sch}
	H\left(\varphi|_\ell\right)=-\left(\varphi|_\ell\right)^{''}+q_\ell\varphi|_{\ell}.
\end{equation}
The domain of $H$ consists of functions on $\Gamma$ which satisfy some natural properties on each half-line in order for (\ref{eq_loc_sch}) to make sense, along with a boundary condition at the gluing point. We give a general description of the analysis done in \cite{SW}. Assume that $\Gamma$ consists of the gluing of $\ell_1,\ldots,\ell_n$. For any $1\leq i\leq n$, define $H_i$ as the Schr\"{o}dinger operator on $\ell_i$ which acts by $H\varphi=-\varphi''+q_{\ell_i}\varphi$, along with a Dirichlet boundary condition at the origin. Denote by $\mu_i$ a scalar spectral measure for $H_i$, and by $\left(\mu_i\right)_s$  the singular part of $\mu_i$ with respect to the Lebesgue measure. Denote also by $P$ the projection-valued spectral measure of $H$, and by $P_s$ its singular part with respect to the Lebesgue measure. The main result of \cite{SW} essentially says the following. A support for $P_s$ can be given by $S=S_1\sqcup S_2$, where $S_1$ consists of energies for which at least two of the singular parts $\left(\mu_1\right)_s,\ldots,\left(\mu_n\right)_s$ overlap, and $P|_{S_2}$ is mutually singular with respect to each of $\left(\mu_1\right)_s,\ldots,\left(\mu_n\right)_s$. In addition, the multiplicity on $S_1$ is equal to the number of overlaps, and the multiplicity on $S_2$ is $1$.

We now describe the discrete analogue of this result, which can be obtained with our approach. Let $G$ be a discrete star-graph, and let $J$ be a Jacobi matrix on $G$. Assume that $G$ consists of $n$ half-lines, and let $v$ be the single vertex of $G$ which satisfies $\deg\left(v\right)=n$. Denote the neighbors of $v$ by $v_1,\ldots,v_n$. For every $1\leq k\leq n$, denote by $G_k$ the connected component of $v_k$ in the graph $G\setminus\left(v,v_k\right)$, and by $J_k$ the operator $P_{G_k}JP_{G_k}$. $J_k$ is simply $J$ restricted to $G_k$ along with a Dirichlet boundary condition at the origin. Denote by $\rho_k$ the spectral measure of $\delta_{v_k}$ w.r.t.\ $J_k$, and by $A_k$ the support of $\left(\rho_k\right)_s$ given by Theorem \ref{gilpe_th}. Finally, let $\mu$ be the scalar spectral measure of $J$ as defined in Section \ref{per_slike_subsection}, and by $S$ the support of $\mu_s$ given by Theorem \ref{main_thm}.
\begin{theorem}
	Denote by $S_1$ the set of real numbers for which there exist at least two indices $1\leq i<j\leq n$ for which $A_i\cap A_j\neq\emptyset$, and denote $S_2=\left(S_1\right)^c$. Then
	\begin{enumerate}
		\item $S_1\subseteq S$ and $N_H|_{S_1}\leq n-1$ for $\mu_s$-almost every $E\in\mathbb{R}$.
		\item $N_{J}|_{S_2\cap S}=1$ for $\mu_s$-almost every $E\in\mathbb{R}$.
	\end{enumerate}
\end{theorem}
\begin{proof}
	\begin{enumerate}
		\item Let $E\in S_1$. w.l.o.g.\ assume that $E\in A_1\cap A_2$. Let $\varphi_1$ be the subordinate solution to $J_1\varphi=E\varphi$ which satisfies $\varphi_1\left(v_1\right)=1$, and let $\varphi_2$ be the subordinate solution to $J_2\varphi=E\varphi$ which satisfies $\varphi_2\left(v_2\right)=-1$.
		Define $\varphi:G\to\mathbb{C}$ by
		\begin{center}
			$\varphi\left(u\right)=\begin{cases}
				\varphi_j\left(u\right) & u\in G_j,\,1\leq j\leq 2\\
				0 & otherwise
			\end{cases}$.
		\end{center}
		It is not hard to verify that $\varphi$ satisfies $H\varphi=E\varphi$, and that it is a subordinate solution, and so $E\in S$. For the second part, note that each subordinate solution $\psi$ must satisfy $\psi\left(0\right)=0$ due to the Dirichlet boundary condition which it satisfies on $G_1$. Thus, the eigenvalue equation at $v$ is given by
		\begin{center}
			$0=E\psi\left(0\right)=\sum\limits_{j=1}^{n}\psi\left(v_j\right)$
		\end{center}
		and so there are at most $n-1$ subordinate solutions. Now, the result follows from Theorem \ref{mul_thm}.
		\item Let $E\in S\cap S_2$, and let $\psi$ be a non-trivial subordinate solution of $H\varphi=E\varphi$. Assume that $\psi$ does not vanish on $G_1$. We claim that $\psi\left(0\right)\neq 0$. Indeed, if $\psi\left(0\right)=0$, then there must be at least one more index $j\neq 1$ such that $\psi$ does not vanish on $G_j$. Otherwise, we have
		\begin{center}
			$0=E\psi\left(0\right)=\sum\limits_{j=1}^{n}\psi\left(v_j\right)=\psi\left(v_1\right)$
		\end{center}
		and it follows that $\psi\equiv 0$. Thus, assume w.l.o.g.\ that $\psi\left(v_2\right)\neq 0$. This implies that $E\in A_1\cap A_2$, which contradicts our assumption that $E\in S_2$. Now, by uniqueness of the subordinate solution on each half-line, we get that $\psi\left(v\right)$ determines $\psi\left(v_j\right)$ for any $1\leq j\leq n$, and so $\dim S\left(E\right)=1$, which implies that $N_J\left(E\right)=1$ for $\mu_s$-almost every $E\in S_2\cap S$, as required.
	\end{enumerate}
\end{proof}

\subsection{Strict inequality in Theorem \ref{mul_thm}}
Consider a star-like graph $G$ for which $C$ is a triangle graph, and there is a half-line attached to every vertex of $C$ (see Figure \ref{sg_fig} (a)). As before, denote the vertices of $C$ by $v_1,v_2,v_3$ and the half-line with $v_i$ as an origin by $G_i$. It is well known (see, e.g.\, \cite[Chapter 7]{Ber}) that there is one-to-one correspondence between probability measures with infinite and bounded support and bounded Jacobi matrices on $\mathbb{N}$, and so in order to construct an example, we start by constructing the appropriate measures. Let $\mu_1$ be a probability measure on $\left[0,1\right]$ such that $\mu_1\left(\left\{0\right\}\right)>0$, and let $\mu_2$ be defined by $d\mu_2=\frac{1}{c}fd\lambda$, where $f(x)=\begin{cases}
	\frac{1}{\sqrt{x}} & 0<x<1\\
	0 & \text{otherwise}
\end{cases}$ and $c=\int_0^1\frac{dx}{\sqrt{x}}$. Clearly, $\mu_2$ is a probability measure on $\left[0,1\right]$ which is absolutely continuous w.r.t.\ $\lambda$. It is not hard to show that the Borel transform of $\mu_2$ satisfies $\left|m_{\mu_2}\left(0+i0\right)\right|=\infty$. For $\theta=\frac{\pi}{4}$ and for $k=1,2$, let $J^{\left(k\right)}$ be a Jacobi matrix on $\mathbb{N}$ such that the spectral measure of $J^{\left(k\right)}_\theta$ is $\mu_k$. We define a Jacobi matrix $J$ on $G$ in the following way: $J|_{G_1}=J|_{G_3}=J^{\left(1\right)}$, $J|_{G_2}=J^{{\left(2\right)}}$, and for every $e\in E\left(C\right)$, $a_e=1$. The fact that $\mu_1\left(\left\{0\right\}\right)>0$ implies that there exists $0\neq\varphi\in\ell^2\left(\mathbb{N}\right)$ such that $J^{\left(1\right)}_\theta\varphi_1=0$ and $\varphi_1\left(1\right)=1$. By Theorem \ref{gilpe_th}, any solution to $J^{\left(2\right)}\varphi=0$ which satisfies $\varphi\left(0\right)=-\varphi\left(1\right)$ is subordinate. We denote by $\varphi_2$ the solution which satisfies $\varphi_2\left(1\right)=-1$. Now, it is can be verified that $\psi_1:G\to\mathbb{C}$ which is defined by
\begin{center}
	$\psi|_{G_1}=\varphi_1,\,\psi|_{G_2}=-\varphi_1,\,\psi|_{G_3}=0$
\end{center}
satisfies $J\psi=0$. In addition, any other non-trivial eigenvector must be a multiple of $\psi$, and so $0$ is a simple eigenvalue. On the other hand, the solution $\widetilde{\psi}$ which is defined by
\begin{center}
	$\widetilde{\psi}|_{G_1}=\varphi_1,\,\widetilde{\psi}|_{G_2}=0,\,\widetilde{\psi}|_{G_3}=\varphi_2$
\end{center}
is also a subordinate solution which is linearly independent of $\psi$, and so $N_J\left(0\right)=1<2=\dim S\left(0\right)$.


\begin{thebibliography}{1}
	
	\bibitem{AW} M.~Aizenman and S.~Warzel, {\it Random operators: Disorder Effects on Quantum Spectra and Dynamics}, volume 168 of {\it Graduate Studies in Mathematics}, Amer.\ Math.\ Soc., Providence, RI, 2015.
	
	\bibitem{Ber} Ju.~M.~Berezanski\u{i}, {\it Expansions in Eigenfunctions of Self-Adjoint Operators}, Transl.\ Math.\ Mono. \textbf{17}, Amer.\ Math.\ Soc., Providence, RI, 1968.
	
	\bibitem{BS} M.~S.~Birman and M.~Z.~Solomyak, {\it Spectral theory of self-adjoint operators in Hilbert space}, Kluwer, Dordrecht, 1987.
	
	\bibitem{Bu} D.~Buschmann, {\it A proof of the Ishii-Pastur theorem by the method of subordinacy}, Univ.\ Iagel.\ Acta Math. \textbf{34} (1997), 29--34.
	
	\bibitem{DKL} D.~Damanik, R.~Killip and D.~Lenz, {\it Uniform spectral properties of one-dimensional quasicrystals, III. $\alpha$-continuity}, Comm.\ Math.\ Phys. \textbf{212} (2000), 191--204.
	
	\bibitem{GT} F.~Gesztesy and E.~Tsekanovskii, {\it On matrix-valued Herglotz functions}, Math. Nachr. \textbf{215} (2000), 61--138.
	
	\bibitem{G1} D.~J.~Gilbert, {\it On subordinacy and spectral multiplicity for a class of singular differential operators}, Proc.\ Roy.\ Soc.\ Edinburgh Sect.\ A \textbf{128} (1998), 549--584.
	
	\bibitem{G2} D.~J.~Gilbert, {\it On subordinacy and analysis of the spectrum of Schr\"{o}dinger operators with two singular endpoints}, Proc.\ Roy.\ Soc.\ Edinburgh Sect.\ A \textbf{1989}, 213--229.
	
	\bibitem{GP} D.~J.~Gilbert and D.~B.~Pearson, {\it On subordinacy and analysis of the spectrum of one-dimensional Schr\"{o}dinger operators}, J.\ Math.\ Anal.\ Appl. \textbf{128} (1987), 30--56.
	
	\bibitem{JaL} V.~Jaksic, Y.~Last, {\it A new proof of Poltoratskii's theorem}, J.\ Funct.\ Anal. \textbf{215} (2004), 103--110.
		
	\bibitem{JL1} S.~Jitomirskaya and Y.~Last, {\it Power law subordinacy and singular spectra, I. Half-line operators}, Acta Math. \textbf{183} (1999), 171--189.
	
	\bibitem{JL2} S.~Jitomirskaya and Y.~Last, {\it Power law subordinacy and singular spectra, II. Line operators}, Comm.\ Math.\ Phys. \textbf{211} (2000), 643--658.
	
	\bibitem{Kac} I.~S.~Kac, {\it Spectral multiplicity of a second-order differential operator and expansion in eigenfunction. (Russian)}, Izv.\ Akad.\ Nauk SSSR Ser.\ Mat. \textbf{27} (1963), 1081--1112.
	
	\bibitem{KP} S.~Khan and D.~B.~Pearson, {\it Subordinacy and spectral theory for infinite matrices}, Helv.\ Phys.\ Acta \textbf{65} (1992), 505--527.
	
	\bibitem{KLS} A.~Kiselev, Y.~Last and B.~Simon, {\it Modified Pr\"{u}fer and EFGP transforms and the spectral analysis of one-dimensional Schr\"{o}dinger operators}, Comm.\ Math.\ Phys. \textbf{194} (1998), 1--45.
	
	\bibitem{Pol} A.~G.~Poltoratskii, {\it The boundary behavior of pseudocontinuable functions}, St. Petersburg Math. J. \textbf{5} (1994), 389--406.
	
	\bibitem{Rem} C.~Remling, {\it Embedded singular continuous spectrum for one-dimensional Schr\"{o}dinger operators}, Trans.\ Amer.\ Math.\ Soc. \textbf{351} (1999), 2479--2497.
	
	\bibitem{Sim3} B.~Simon, {\it On a theorem of Kac and Gilbert}, J.\ Funct.\ Anal. \textbf{223} (2005), 109--115.
	
	\bibitem{Sim} B.~Simon, {\it Spectral analysis of rank one perturbations and applications}, in “Proc. Mathematical Quantum Theory, II: Schr¨odinger Operators” (Vancouver, Canada, 1993), CRM Proceedings and Lecture Notes, 8, American Mathematical Society, Providence, RI, 1995, 109–149.
	
	\bibitem{Sim2} B.~Simon, {\it Szego's theorem and its descendants}, M.B.~ Porter Lectures, Princeton University Press, Prinston, NJ, 2011.
	
	\bibitem{SW} S.~Simonov and H.~Woracek, {\it Spectral multiplicity of selfadjoint Schr\"{o}dinger operators on star-graphs with standard interface conditions}, Integral Equations Operator Theory \textbf{78} (2014), 523--575.
	
	\bibitem{Z} A.~Zlato\v{s}, {\it Sparse potentials with fractional Hausdorff dimension}, J.\ Funct.\ Anal. \textbf{207} (2004), 216--252.
	
\end{thebibliography}
\end{document}